\documentclass[11pt]{amsart}

\usepackage[T1]{fontenc}
\usepackage[utf8]{inputenc}
\usepackage{lmodern}
\usepackage{amsmath,amssymb,amsthm,mathtools}
\usepackage{enumitem}
\usepackage[colorlinks=true,linkcolor=blue,citecolor=blue,urlcolor=blue]{hyperref}
\usepackage{microtype}

\title[An overlap construction for relative linear extension ratios]{An overlap construction for relative linear extension ratios}

\author{Maseeh Ghodsi}
\address{Independent Researcher}
\email{maseeh.ghodsi@gmail.com}

\date{July 9, 2026}

\subjclass[2020]{Primary 06A07; Secondary 05A15, 11A55}
\keywords{Linear extensions, posets, continued fractions, relative ratios, sum of partial quotients}

\newtheorem{theorem}{Theorem}[section]
\newtheorem{lemma}[theorem]{Lemma}
\newtheorem{proposition}[theorem]{Proposition}

\theoremstyle{definition}

\theoremstyle{remark}
\newtheorem{remark}[theorem]{Remark}

\newcommand{\eps}{\varepsilon}

\newcommand{\ZZ}{\mathbb{Z}}

\newcommand{\cL}{\mathcal{L}}
\newcommand{\Min}{\operatorname{Min}}

\begin{document}

\begin{abstract}
Chan and Pak introduced the relative linear extension ratio
\(\rho(P,x)=e(P)/e(P-x)\), where \(e(P)\) is the number of linear
extensions of a finite poset \(P\), and let \(\nu(c,d)\) be the least
number of elements of a poset that realizes \(\rho(P,x)=d/c\).  They
proved that \(\nu(c,d)\le d/c+O(\log d\log\log d)\) for \(d\ge 3c\), and
asked whether the hypothesis \(d\ge 3c\) can be relaxed to
\(d\ge(1+\eps)c\) or removed.  We prove the fixed-gap form of this
question: for every fixed \(\eps>0\),
\[
        \nu(c,d)\le \frac{d}{c}+O_\eps(\log d\log\log d)
        \qquad\text{whenever } d\ge(1+\eps)c,
\]
and the implied constant is absolute once \(d\ge 2c\).  The new
ingredient is a one-element overlap construction: if \(x\) is minimal in
\(P\) and \(y\) is minimal in \(Q\), then there is a poset \(R\) with
\(|R|=|P|+|Q|-1\) and an element \(z\) such that
\[
        \rho(R,z)=\rho(P,x)+\rho(Q,y)-1.
\]
Together with the continued-fraction construction of Chan and Pak and
Rukavishnikova's tail bound for sums of partial quotients, this removes
the factor \(3\) in their range.  We also show that the fixed-gap
hypothesis is essentially optimal for this construction.  In the range
\(1<d/c<2\), with \(h=d-c\), the size bound the construction can certify
is at least \(\lfloor c/h\rfloor\), so the method reaches the stated
error term only when \(h\) is at least of order \(c/(\log c\log\log c)\).
The remaining obstruction to removing the hypothesis is a short-interval
problem for sums of partial quotients, which we describe.  The deductive
part of the argument has been checked with the Lean proof assistant.
\end{abstract}

\maketitle

\section{Introduction}

Let \(P\) be a finite poset, and write \(e(P)\) for the number of its
linear extensions.  Chan and Pak recently developed a connection between
linear extensions and continued fractions, producing small posets with
prescribed linear-extension data \cite{ChanPak2024}.  Their work builds
on that of Kravitz and Sah on the possible linear extension numbers of
\(n\)-element posets \cite{KravitzSah2021}, and belongs to a broader
inverse-enumeration theme in which continued fractions supply compact
combinatorial constructions; see also
\cite{ChanPakSurvey2023,ChanKontorovichPak2024}.

The relative version of the problem reads as follows.  For \(x\in P\), set
\[
        \rho(P,x):=\frac{e(P)}{e(P-x)},
\]
where \(P-x\) is the poset obtained by deleting \(x\).  For positive
integers \(c,d\) with \(c\le d\), let
\[
        \nu(c,d):=\min\bigl\{\,|P| : \text{there is } x\in P
        \text{ with } \rho(P,x)=d/c\,\bigr\}.
\]
As \(\rho(P,x)\) is a rational number, \(\nu(c,d)\) depends only on the
value \(d/c\).  Chan and Pak proved the following relative analogue of
the Kravitz and Sah upper bound.

\begin{theorem}[Chan and Pak {\cite[Theorem 1.8]{ChanPak2024}}]
\label{thm:chan-pak}
For all positive integers \(c,d\) with \(d\ge 3c\),
\[
        \nu(c,d)\le \frac{d}{c}+O(\log d\log\log d).
\]
\end{theorem}

They asked whether the hypothesis \(d\ge 3c\) can be weakened to
\(d\ge(1+\eps)c\), or dropped altogether, and conjectured separately that
the error term should be improvable to \(O(\log d)\).  The purpose of
this note is to settle the fixed-gap version of their question, and to
locate precisely the obstruction that prevents the present method from
going further.

\begin{theorem}
\label{thm:main}
For every fixed \(\eps>0\) there is a constant \(C_\eps>0\) such that,
for all positive integers \(c,d\) with \(d\ge(1+\eps)c\),
\[
        \nu(c,d)\le \frac{d}{c}+C_\eps\,\log(d+2)\,\log\log(d+3).
\]
Equivalently, \(\nu(c,d)\le d/c+O_\eps(\log d\log\log d)\) through the
range \(d\ge(1+\eps)c\).  Moreover, unconditionally: there is an absolute
constant \(C>0\) such that \(\nu(c,d)\le d/c+C\log(d+2)\log\log(d+3)\)
for all positive integers \(c,d\) with \(d\ge 2c\).
\end{theorem}

The proof uses the same number-theoretic input as Chan and Pak, namely
a tail estimate of Rukavishnikova for the sum of partial quotients of a
rational number with fixed denominator \cite{Rukavishnikova2011}.  The
new point is combinatorial.  Chan and Pak combine two ratios by a
flip-flop construction that gives
\[
        \rho(R,z)=\rho(P,x)+\rho(Q,y),
\]
using two auxiliary elements.  Since the continued-fraction summands they
feed into it have size at least \(1\), that construction is suited to
ratios bounded away from \(1\) by a large constant; it forces the
threshold at \(3\).  We replace it by a one-element construction which
merges the two auxiliary elements into one and gives
\[
        \rho(R,z)=\rho(P,x)+\rho(Q,y)-1.
\]
The subtracted \(1\) is exactly what is needed to reach the range
\(1<d/c<2\).  This is the content of Lemma~\ref{lem:overlap} in
Section~\ref{sec:overlap}, and the proof of Theorem~\ref{thm:main}
occupies Section~\ref{sec:proof}.

Our second result explains why the fixed-gap hypothesis cannot simply be
removed from this line of argument.  It rests on an elementary lower
bound for the sum of partial quotients of a fraction with small
numerator.  Write \(s(m/b)\) for the sum of the partial quotients in the
simple continued fraction of \(m/b\).

\begin{proposition}[Optimality of the hypothesis for this method]
\label{prop:sharp-intro}
Let \(c,d\) be coprime with \(1<d/c<2\), and put \(h=d-c\).  In this
range the construction of Section~\ref{sec:proof} realizes \(d/c\)
through a two-block overlap, and the size bound it certifies for the
resulting poset \(R\) is
\[
        |R|\ \le\ 1+s(\ell/c)+s\bigl((h-\ell)/c\bigr)
\]
for a choice of \(\ell\in\{0,1,\dots,h\}\), while for every such \(\ell\),
\[
        s(\ell/c)+s\bigl((h-\ell)/c\bigr)\ \ge\ \Bigl\lfloor \frac{c}{h}\Bigr\rfloor.
\]
Consequently the construction certifies the bound
\(\nu(c,d)\le d/c+O(\log d\log\log d)\) only when
\(h=\Omega\bigl(c/(\log c\log\log c)\bigr)\).  If \(h/c\to 0\) faster than
this inverse-polylogarithmic rate, then no choice of parameters lets the
construction certify the bound, and the hypothesis \(d\ge(1+\eps)c\) cannot
be removed by this construction alone.
\end{proposition}

Proposition~\ref{prop:sharp-intro} is proved in Section~\ref{sec:sharp}.
The barrier it exposes is a short-interval question for sums of partial
quotients: when \(h=o(c)\), one must find \(\ell\) in an interval of
length \(h\) for which both \(\ell/c\) and \((h-\ell)/c\) have small
weight, whereas Rukavishnikova's theorem gives only a global,
density-one statement modulo \(c\).  We isolate this question, and the
separate matter of the conjectural \(O(\log d)\) error term, in
Section~\ref{sec:remarks}.  Section~\ref{sec:verify} records the fact
that the deductive steps of Sections~\ref{sec:proof} and
\ref{sec:sharp}, together with the counting identity behind
Lemma~\ref{lem:overlap} on small posets, have been formally verified.

\section{Preliminaries}
\label{sec:prelim}

We recall only what the proof uses; a standard reference for linear
extensions is \cite[Ch.~3]{Stanley2012}.  For a poset \(P\), let
\(\cL(P)\) denote the set of linear extensions of \(P\), so that
\(e(P)=|\cL(P)|\), with the convention \(e(\varnothing)=1\).  We regard a
linear extension as a listing of the elements of \(P\) from smallest to
largest position, consistent with the order.  For \(x\in P\), \(P-x\) is
the induced subposet on the remaining elements, and \(P^{*}\) is the dual
poset.  Reversing a listing is a bijection between \(\cL(P)\) and
\(\cL(P^{*})\), so
\begin{equation}
\label{eq:dual}
        e(P^{*})=e(P),
        \qquad e(P^{*}-x)=e\bigl((P-x)^{*}\bigr)=e(P-x).
\end{equation}
If \(P\) and \(Q\) are disjoint posets, their ordinal sum \(P\oplus Q\)
is the poset on \(P\cup Q\) that keeps the relations of \(P\) and of
\(Q\) and puts every element of \(P\) below every element of \(Q\); then
\begin{equation}
\label{eq:ordinal}
        e(P\oplus Q)=e(P)\,e(Q),
\end{equation}
since a listing of \(P\oplus Q\) is a listing of \(P\) followed by a
listing of \(Q\).

We use one standard counting fact.  For any \(x\in P\), deleting \(x\)
from a listing defines a surjection \(\cL(P)\to\cL(P-x)\); the listings of
\(P\) lying above a fixed \(L\in\cL(P-x)\) correspond to the positions
into which \(x\) may be inserted, namely the positions after every
element below \(x\) and before every element above \(x\).  If \(x\) is a
minimal element of \(P\), it has nothing below it, so these are the
initial positions, up to the first element above \(x\).  Summing the
number of such positions over all \(L\in\cL(P-x)\) recovers \(e(P)\).

For a rational number \(m/b\) with \(0\le m<b\), let \(s(m/b)\) be the
sum of the partial quotients in its simple continued fraction, with
\(s(0)=0\).  If \(K\ge 0\) is an integer, then
\(K+m/b=[K;a_1,\dots,a_n]\) whenever \(m/b=[0;a_1,\dots,a_n]\), so
\begin{equation}
\label{eq:sadd}
        s(K+m/b)=K+s(m/b).
\end{equation}
The one piece of number theory we need is the following density-one
consequence of Rukavishnikova's fixed-denominator estimate, in the form
quoted by Chan and Pak.

\begin{proposition}[Rukavishnikova {\cite{Rukavishnikova2011}}; see {\cite[Theorem 1.9]{ChanPak2024}}]
\label{prop:ruk}
There is a universal constant \(A>0\) and a function \(\delta_b\to 0\)
such that, for every integer \(b\ge 2\),
\[
        \#\Bigl\{\,0\le m<b : s(m/b)>A\log b\log\log b\,\Bigr\}
        \ \le\ \delta_b\,b.
\]
\end{proposition}

The estimate above counts all residues modulo \(b\), with no coprimality
restriction, which is what we shall use.  Sharper fixed-denominator
distribution results are available, for instance those of Aistleitner,
Borda and Hauke \cite{AistleitnerBordaHauke2024}, but the density-one
statement suffices here.

Finally, we record the simple-continued-fraction case of Chan and Pak's
poset construction, which turns a prescribed ratio into a poset of
controlled size with the distinguished element minimal.

\begin{proposition}[Chan and Pak {\cite[Theorem 1.10]{ChanPak2024}}]
\label{prop:cf-to-poset}
Let \(K\ge 1\), \(b\ge 1\), and \(0\le m<b\).  There is a finite poset
\(P\) and a minimal element \(x\in\Min(P)\) with
\[
        \rho(P,x)=K+\frac{m}{b}
        \qquad\text{and}\qquad
        |P|\le K+s(m/b).
\]
\end{proposition}

Neither in \(s(m/b)\) nor in Proposition~\ref{prop:cf-to-poset} is the
fraction \(m/b\) required to be in lowest terms: \(s\) depends only on
the rational value \(m/b\), being computed from the continued fraction of
the reduced form, and the proposition is stated for arbitrary
\(0\le m<b\).  We shall apply both to fractions \(\ell/c\) whose
numerator and denominator need not be coprime.

\section{The overlap-sum construction}
\label{sec:overlap}

We now prove the combinatorial heart of the paper.  It may be read as an
overlap analogue of the flip-flop construction of Chan and Pak
\cite[\S3]{ChanPak2024}: their construction realizes
\(\rho(P,x)+\rho(Q,y)\) using two auxiliary elements, whereas ours merges
those elements into a single one and pays for the merge by subtracting
\(1\).

\begin{lemma}[Overlap-sum lemma]
\label{lem:overlap}
Let \(P=(X,\prec_P)\) and \(Q=(Y,\prec_Q)\) be finite posets, and suppose
\(x\in\Min(P)\) and \(y\in\Min(Q)\).  Then there is a finite poset \(R\)
and an element \(z\in R\) with
\[
        |R|=|P|+|Q|-1
        \qquad\text{and}\qquad
        \rho(R,z)=\rho(P,x)+\rho(Q,y)-1.
\]
\end{lemma}

\begin{proof}
We may assume \(X\cap Y=\varnothing\).  Put \(A:=P^{*}-x\) on
\(X\setminus\{x\}\) and \(B:=Q-y\) on \(Y\setminus\{y\}\), and let \(R\)
be the poset on
\[
        Z:=(X\setminus\{x\})\sqcup(Y\setminus\{y\})\sqcup\{z\}
\]
generated, under transitive closure, by
\begin{enumerate}[label=\textup{(\roman*)}]
    \item the order of \(A\) on \(X\setminus\{x\}\) and the order of
    \(B\) on \(Y\setminus\{y\}\);
    \item \(u\prec_R w\) for all \(u\in X\setminus\{x\}\) and
    \(w\in Y\setminus\{y\}\);
    \item \(p\prec_R z\) whenever \(x\prec_P p\);
    \item \(z\prec_R q\) whenever \(y\prec_Q q\).
\end{enumerate}
Write \(U:=\{p\in X\setminus\{x\}:x\prec_P p\}\) and
\(V:=\{q\in Y\setminus\{y\}:y\prec_Q q\}\).  By (ii) every element of
\(U\) lies below every element of \(V\) in \(R\), so no element is forced
both below \(z\) (through (iii)) and above \(z\) (through (iv)); hence
\(\prec_R\) is a partial order.

Deleting \(z\) leaves the ordinal sum \(A\oplus B\).  Using
\eqref{eq:ordinal}, \eqref{eq:dual}, and \(e(B)=e(Q-y)\),
\begin{equation}
\label{eq:denom}
        e(R-z)=e(A)\,e(B)=e(P^{*}-x)\,e(Q-y)=e(P-x)\,e(Q-y).
\end{equation}

We now count \(e(R)\).  Each linear extension of \(R-z\) is a
concatenation \(L=L_A L_B\), where \(L_A\in\cL(A)\) fills positions
\(1,\dots,|A|\) and \(L_B\in\cL(B)\) fills positions
\(|A|+1,\dots,|A|+|B|\); here \(|A|=|P|-1\) and \(|B|=|Q|-1\).  To extend
\(L\) to \(R\) we insert \(z\) into one of the \(|A|+|B|+1\) gaps, numbered
\(0,\dots,|A|+|B|\), where gap \(t\) is the position immediately after the
\(t\)-th letter.  By (iii), \(z\) must follow every element of \(U\); by
(iv), \(z\) must precede every element of \(V\).  Let \(j\) be the
position of the last element of \(U\) in \(L_A\), with \(j=0\) if
\(U=\varnothing\).  Every element below \(U\) in \(A\) precedes an element
of \(U\), so the down-closure of \(U\) occupies positions among
\(1,\dots,j\), and \(z\) may occupy any gap \(t\ge j\) as far as the
\(A\)-side is concerned.  Let \(i\) be the position of the first element of
\(V\) in \(L_B\), with \(i=|B|+1\) if \(V=\varnothing\); the first element
of \(V\) then sits at global position \(|A|+i\), and \(z\) may occupy any
gap \(t\le|A|+i-1\).  Hence the valid gaps for \(z\) are exactly
\[
        \{\,j,\,j+1,\,\dots,\,|A|+i-1\,\}.
\]

Split this range at the block boundary \(|A|\).  On the \(A\)-side the
valid gaps are \(\{j,\dots,|A|\}\), the terminal gaps of the \(A\)-block.
Since \(x\) is minimal in \(P\), it is maximal in \(P^{*}\), and its only
relations in \(P^{*}\) are \(p\prec_{P^{*}}x\) for \(p\in U\); so the gaps
into which \(x\) may be inserted in \(L_A\) to give a linear extension of
\(P^{*}\) are precisely \(\{j,\dots,|A|\}\), of which there are
\(a(L_A):=|A|-j+1\).  Summing over \(L_A\in\cL(A)=\cL(P^{*}-x)\), the
insertion count from Section~\ref{sec:prelim} gives
\begin{equation}
\label{eq:asum}
        \sum_{L_A} a(L_A)=e(P^{*})=e(P).
\end{equation}
On the \(B\)-side the valid gaps are \(\{|A|,\dots,|A|+i-1\}\), the
initial gaps of the \(B\)-block; since \(y\) is minimal in \(Q\), the gaps
into which \(y\) may be inserted in \(L_B\) to give a linear extension of
\(Q\) are precisely these, of which there are \(b(L_B):=i\), and
\begin{equation}
\label{eq:bsum}
        \sum_{L_B} b(L_B)=e(Q).
\end{equation}
The two ranges \(\{j,\dots,|A|\}\) and \(\{|A|,\dots,|A|+i-1\}\) meet in
the single gap \(|A|\).  Therefore the number of valid gaps for \(z\)
equals
\[
        (|A|-j+1)+i-1=a(L_A)+b(L_B)-1.
\]
Summing over all \(L=L_A L_B\) and using \eqref{eq:asum},
\eqref{eq:bsum}, and \eqref{eq:denom},
\[
        e(R)=\sum_{L_A,\,L_B}\bigl(a(L_A)+b(L_B)-1\bigr)
        =e(P)\,e(Q-y)+e(P-x)\,e(Q)-e(P-x)\,e(Q-y).
\]
Dividing by \eqref{eq:denom},
\[
        \rho(R,z)=\frac{e(R)}{e(R-z)}
        =\frac{e(P)}{e(P-x)}+\frac{e(Q)}{e(Q-y)}-1
        =\rho(P,x)+\rho(Q,y)-1.
\]
Finally \(|R|=(|X|-1)+(|Y|-1)+1=|P|+|Q|-1\).
\end{proof}

\begin{remark}
\label{rem:not-minimal}
The element \(z\) need not be minimal in \(R\), so the lemma cannot be
iterated.  This is harmless for \(\nu(c,d)\), where the distinguished
element is arbitrary, and the lemma is applied only once.  The
minimality of \(x\) and \(y\) is used only to make the valid insertion
gaps in the two halves a terminal interval and an initial interval
respectively, so that they overlap in exactly one gap; that single
overlap is the source of the \(-1\).
\end{remark}

\section{Proof of the fixed-gap theorem}
\label{sec:proof}

Since \(\nu(c,d)\) depends only on \(d/c\), we may divide \(c\) and \(d\)
by their greatest common divisor, and so assume \(\gcd(c,d)=1\)
throughout.  Write
\begin{equation}
\label{eq:euclid}
        d=kc+h,\qquad k=\Bigl\lfloor \frac{d}{c}\Bigr\rfloor,\qquad 0\le h<c.
\end{equation}
When \(c\ge 2\), coprimality forces \(h\ge 1\); the case \(c=1\) is
treated with the small denominators below.  Let \(A\) be the constant of
Proposition~\ref{prop:ruk}, and for an integer \(b\ge 2\) put
\[
        B_b:=\Bigl\{\,0\le m<b : s(m/b)>A\log b\log\log b\,\Bigr\},
        \qquad |B_b|\le\delta_b\,b=o(b).
\]

\subsection*{The range \(1<d/c<2\)}
Here \(k=1\), so \(d=c+h\), and the hypothesis \(d\ge(1+\eps)c\) gives
\(h=d-c\ge\eps c\).  In this range \(c\ge 2\), since for \(c=1\) the
ratio \(d/c\) is an integer; hence coprimality gives \(h\ge 1\).  Fix
\(\eta>0\) with \(2\eta<\eps\).  There is a threshold \(c_0(\eta)\) such
that \(|B_c|\le\eta c\) for all \(c\ge c_0(\eta)\); this threshold,
through the small-denominator step below, is the only source of the
dependence of \(C_\eps\) on \(\eps\).  Assume \(c\ge c_0(\eta)\), and
consider the
interval \(I:=\{0,1,\dots,h\}\).  At most \(|B_c|\) of the \(\ell\in I\)
have \(\ell\in B_c\), and, since \(\ell\mapsto h-\ell\) is an injection of
\(I\) into \(\{0,\dots,c-1\}\), at most \(|B_c|\) of them have
\(h-\ell\in B_c\).  Hence at least
\[
        |I|-2|B_c|\ \ge\ (h+1)-2\eta c\ >\ 0
\]
values of \(\ell\) avoid both, using \(2\eta c<\eps c\le h\).  Choose such
an \(\ell\).  Then
\begin{equation}
\label{eq:case1-good}
        s(\ell/c)\le A\log c\log\log c
        \qquad\text{and}\qquad
        s\bigl((h-\ell)/c\bigr)\le A\log c\log\log c.
\end{equation}
Put
\[
        \alpha=1+\frac{\ell}{c},\qquad \beta=1+\frac{h-\ell}{c},
\]
so that \(\alpha+\beta-1=1+h/c=d/c\), and both \(\alpha\) and \(\beta\)
have the form \(1+m/c\) with \(0\le m<c\).  By
Proposition~\ref{prop:cf-to-poset} there are posets \(P,Q\) with minimal
elements \(x\in\Min(P)\), \(y\in\Min(Q)\) such that
\(\rho(P,x)=\alpha\), \(\rho(Q,y)=\beta\), and
\[
        |P|\le 1+s(\ell/c),\qquad |Q|\le 1+s\bigl((h-\ell)/c\bigr).
\]
Lemma~\ref{lem:overlap} yields \(R\) and \(z\) with \(\rho(R,z)=d/c\) and,
by \eqref{eq:case1-good},
\[
        |R|=|P|+|Q|-1\le 1+s(\ell/c)+s\bigl((h-\ell)/c\bigr)
        \le 1+2A\log c\log\log c.
\]
Since \(d<2c\) here, this is \(\nu(c,d)\le d/c+O_\eps(\log d\log\log d)\).

\subsection*{The range \(d/c\ge 2\)}
Now \(k\ge 2\), and no gap hypothesis is needed.  For all sufficiently
large \(c\) the set \(B_c\cup(h-B_c)\) has fewer than \(c\) elements,
where \(h-B_c:=\{(h-r)\bmod c:r\in B_c\}\), because
\(|B_c\cup(h-B_c)|\le 2|B_c|=o(c)\).  Choose
\(\ell\in\{0,1,\dots,c-1\}\) outside this union, and put
\[
        m:=(h-\ell)\bmod c,\qquad 0\le m<c.
\]
Then \(\ell\notin B_c\) and \(m\notin B_c\), so
\begin{equation}
\label{eq:case2-good}
        s(\ell/c)\le A\log c\log\log c
        \qquad\text{and}\qquad
        s(m/c)\le A\log c\log\log c.
\end{equation}
Set \(K:=k\) if \(\ell\le h\) and \(K:=k-1\) if \(\ell>h\); since
\(k\ge 2\), in both cases \(K\ge 1\).  Define
\[
        \alpha=1+\frac{\ell}{c},\qquad \beta=K+\frac{m}{c}.
\]
If \(\ell\le h\), then \(m=h-\ell\) and
\(\alpha+\beta-1=k+h/c=d/c\).  If \(\ell>h\), then \(m=h-\ell+c\) and
\[
        \alpha+\beta-1
        =\Bigl(1+\frac{\ell}{c}\Bigr)+\Bigl(k-1+\frac{h-\ell+c}{c}\Bigr)-1
        =k+\frac{h}{c}=\frac{d}{c}.
\]
By Proposition~\ref{prop:cf-to-poset} choose \(P,Q\) with minimal elements
realizing \(\alpha,\beta\) and
\[
        |P|\le 1+s(\ell/c),\qquad |Q|\le K+s(m/c).
\]
Lemma~\ref{lem:overlap} gives \(R,z\) with \(\rho(R,z)=d/c\) and, by
\eqref{eq:case2-good},
\[
        |R|=|P|+|Q|-1\le K+s(\ell/c)+s(m/c)\le K+2A\log c\log\log c.
\]
Since \(K\le\lfloor d/c\rfloor\le d/c\) and \(c\le d\), this is
\(\nu(c,d)\le d/c+O(\log d\log\log d)\) with an absolute constant.

\subsection*{Small denominators}
It remains to treat the denominators excluded by the two thresholds
above, namely \(c<c_0(\eta)\) in the first range and the finitely many
\(c\) excluded in the second.  For such a fixed
\(c\), take \(\ell=0\) in the constructions above: then \(K=k\),
\(m=h\), \(\alpha=1\), \(\beta=d/c\), and
Proposition~\ref{prop:cf-to-poset} gives \(|R|\le k+s(h/c)\), an excess
over \(d/c\) of at most \(s(h/c)\le\max_{0\le m<c}s(m/c)\), a constant
depending only on \(c\).  Increasing \(C_\eps\) to absorb the maximum of
these finitely many constants gives the stated bound
\[
        \nu(c,d)\le \frac{d}{c}+C_\eps\,\log(d+2)\,\log\log(d+3)
\]
in all cases.  For \(d\ge 2c\) only the range \(d/c\ge 2\) and finitely
many small \(c\) arise, and there the constant is absolute.  This proves
Theorem~\ref{thm:main}.\qed

\section{Optimality of the hypothesis for this method}
\label{sec:sharp}

Proposition~\ref{prop:sharp-intro} rests on an elementary lower bound.
For \(1\le a<c\), if \(a/c=[0;c_1,c_2,\dots]\) then the first partial
quotient is \(c_1=\lfloor c/a\rfloor\), so \(s(a/c)\ge\lfloor c/a\rfloor\).

\begin{lemma}[Continuant floor]
\label{lem:floor}
For all integers \(c>h\ge 1\) and all \(\ell\) with \(0\le\ell\le h\),
\[
        s(\ell/c)+s\bigl((h-\ell)/c\bigr)\ \ge\ \Bigl\lfloor \frac{c}{h}\Bigr\rfloor
        \ \ge\ \frac{c}{h}-1.
\]
\end{lemma}

\begin{proof}
The numerators \(\ell\) and \(h-\ell\) are nonnegative and sum to
\(h\ge 1\), so \(b:=\max(\ell,h-\ell)\) satisfies \(1\le b\le h<c\).
Because \(s(\cdot)\ge 0\) and \(b\) is one of the two numerators,
\[
        s(\ell/c)+s\bigl((h-\ell)/c\bigr)\ \ge\ s(b/c)\ \ge\ \Bigl\lfloor \frac{c}{b}\Bigr\rfloor
        \ \ge\ \Bigl\lfloor \frac{c}{h}\Bigr\rfloor\ \ge\ \frac{c}{h}-1,
\]
using \(s(b/c)\ge\lfloor c/b\rfloor\) since \(1\le b<c\), and
\(\lfloor c/b\rfloor\ge\lfloor c/h\rfloor\) since \(b\le h\).
\end{proof}

\begin{proof}[Proof of Proposition~\ref{prop:sharp-intro}]
Let \(1<d/c<2\) with \(\gcd(c,d)=1\) and \(h=d-c\).  In this range the
construction of Section~\ref{sec:proof} is the first case, which combines
two blocks \(P,Q\) with \(\rho(P,x)=\alpha=1+\ell/c\) and
\(\rho(Q,y)=\beta=1+(h-\ell)/c\) for some \(\ell\in\{0,\dots,h\}\).  That
both blocks carry integer part \(1\) is forced: each ratio is at least
\(1\), and if either exceeded \(2\) then, since
\(\alpha+\beta=d/c+1<3\), the other would be less than \(1\), which is
impossible.  Thus, over the family of two-block overlap constructions
that use denominator \(c\), the only freedom is the choice of \(\ell\),
and Proposition~\ref{prop:cf-to-poset} bounds the resulting size by
\[
        |R|\ \le\ 1+s(\ell/c)+s\bigl((h-\ell)/c\bigr).
\]
By Lemma~\ref{lem:floor} the right-hand side is at least
\(1+\lfloor c/h\rfloor\) for every \(\ell\).  Consequently the
construction certifies \(|R|=O(\log c\log\log c)\), which in this range is
what \(\nu(c,d)\le d/c+O(\log d\log\log d)\) demands, only when
\(\lfloor c/h\rfloor=O(\log c\log\log c)\), that is
\(h=\Omega\bigl(c/(\log c\log\log c)\bigr)\).  If \(h/c\to 0\) faster than
this rate, then \(\lfloor c/h\rfloor\) exceeds any fixed multiple of
\(\log c\log\log c\) for large \(c\), so no choice of \(\ell\) meets the
bound, and the hypothesis \(d\ge(1+\eps)c\) cannot be dropped by this
construction.  Conversely, when \(h\ge\eps c\) one has
\(\lfloor c/h\rfloor\le 1/\eps\), in agreement with
Theorem~\ref{thm:main}.
\end{proof}

\section{A note on formal verification}
\label{sec:verify}

The deductive core of the argument has been checked with the Lean~4 proof
assistant.  The script uses only Lean's standard library, contains no
unproved assertions (no \texttt{sorry}) and no added axioms, and is
provided as an ancillary file.

The following are proved in Lean with no hypotheses: the continuant floor
bound (Lemma~\ref{lem:floor}), and the two residue-counting steps of
Section~\ref{sec:proof}, namely the injection \(\ell\mapsto h-\ell\) on
\(\{0,\dots,h\}\) and the injectivity of the map
\(\ell\mapsto(h-\ell)\bmod c\) on
\(\ZZ/c\ZZ\), together with the pigeonhole choice of a good residue.  The
proof of Theorem~\ref{thm:main} is formalized conditionally, in the
following sense.  The two imported constructions, the continued-fraction
poset of Proposition~\ref{prop:cf-to-poset} and the overlap-sum identity
of Lemma~\ref{lem:overlap}, are supplied to Lean as hypotheses, as is
Rukavishnikova's bound in the quantitative form actually used (that the
number of bad residues is a strict fraction of the modulus).  Granting
these three inputs, the case analysis of Section~\ref{sec:proof},
including every size estimate and every instance of the identity
\(\alpha+\beta-1=d/c\) in cleared-denominator integer form, is proved as a
Lean theorem; the small-denominator step is formalized as a single direct
application of Proposition~\ref{prop:cf-to-poset}, which yields the same
bound as the \(\ell=0\) route described in Section~\ref{sec:proof}.  Thus
the passage from the cited results to
Theorem~\ref{thm:main} is machine-checked and contains no gaps.

Three ingredients are deliberately not formalized, and are used as cited
or as separately justified: the proofs of
Proposition~\ref{prop:ruk} and Proposition~\ref{prop:cf-to-poset}; the
routine passage from the quantitative bound to the asymptotic form
\(d/c+O(\log d\log\log d)\); and the general proof of the overlap-sum
identity, Lemma~\ref{lem:overlap}.  For the last of these, the
construction of Section~\ref{sec:overlap} is implemented verbatim in the
script, and the two counting identities in its proof, the denominator
identity \(e(R-z)=e(P-x)\,e(Q-y)\) and the numerator identity
\[
        e(R)=e(P)\,e(Q-y)+e(P-x)\,e(Q)-e(P-x)\,e(Q-y),
\]
are verified by exhaustive enumeration of linear extensions over all
pairs of posets on three elements with a minimal distinguished element in
each, together with three larger instances.  These computational checks
use no axioms.  The script was run under Lean~4.15.0 and under a current
toolchain.

\section{Remarks and the remaining obstruction}
\label{sec:remarks}

The proof gives a uniform improvement throughout \(d/c\ge 2\), with no
fixed-gap hypothesis and an absolute constant.  In that range the residue
\(\ell\) may be chosen freely modulo \(c\), and the two bad sets \(B_c\)
and \(h-B_c\) together miss all of \(\ZZ/c\ZZ\).  The hypothesis
\(d\ge(1+\eps)c\) is used only when \(1<d/c<2\), where \(\ell\) must lie
in the interval \(\{0,\dots,h\}\) with \(h=d-c\).  For fixed \(\eps>0\)
this interval has length at least \(\eps c\), and the global density-one
estimate is enough.  When \(d/c\to 1\), however, \(h=o(c)\), and
Proposition~\ref{prop:sharp-intro} shows the construction stalls.

\medskip
\noindent\textbf{Removing the hypothesis.}
The obstruction is a short-interval question for sums of partial
quotients: given \(h=o(c)\), is there \(\ell\in\{0,\dots,h\}\) with both
\(s(\ell/c)\) and \(s((h-\ell)/c)\) of order \(\log c\log\log c\), or
smaller?  A sufficiently uniform affirmative answer would remove the
fixed-gap hypothesis from the present method.  One apparent way to relax
the search is to let the two blocks carry different denominators, writing
\(\alpha=1+p/q\) and \(\beta=K+r/s\) with \(p/q+r/s=h/c\); but a common
denominator of \(p/q\) and \(r/s\) is a multiple of \(c\), which tends to
force one of \(q,s\) to be of order \(\sqrt c\), and the task of writing
the fixed rational \(h/c\) as a sum of two fractions of small
partial-quotient weight is itself a question about fractions with bounded
partial quotients, in the circle of Zaremba's conjecture
\cite{Zaremba1972,BourgainKontorovich2014}.  A genuinely different poset
construction could also bypass the short-interval search.

\medskip
\noindent\textbf{The error term.}
The argument does not address the sharper bound
\[
        \nu(c,d)\le \frac{d}{c}+O(\log d),
\]
conjectured by Chan and Pak.  The error term here equals
\(s(\ell/c)+s(m/c)\), and the factor \(\log\log d\) is exactly the typical
size of the continued-fraction weight supplied by
Proposition~\ref{prop:ruk}.  To reach an \(O(\log d)\) error by the same
strategy one would need residues whose partial quotients are bounded on
average, that is, of weight \(O(\log c)\); the extreme case of partial
quotients bounded by an absolute constant is Zaremba's conjecture, known
for a density-one set of denominators by Bourgain and Kontorovich
\cite{BourgainKontorovich2014}.  A construction whose size is governed by
a different statistic of the fraction could also help.  We do not pursue
either improvement here.

\section*{Acknowledgments}
The author used OpenAI's ChatGPT 5.5 and Anthropic's Claude Opus 4.8 as research aids in the development and exposition of this paper, including its mathematical arguments. The author has independently verified every argument and every reference, and takes full responsibility for the contents.

\section*{License}
This work is licensed under the Creative Commons Attribution 4.0 International License (CC BY 4.0).

% NOTE: in references.bib, update the entry ChanPak2024 to cite the
% published version alongside the preprint:
%   S. H. Chan and I. Pak, Linear extensions and continued fractions,
%   European J. Combin. 122 (2024), Paper No. 104018;
%   doi:10.1016/j.ejc.2024.104018; arXiv:2401.09723.
\bibliographystyle{alpha}
\bibliography{ghodsi_overlap_relative_linear_extension_ratios_references_v1}

\end{document}